\documentclass[12pt]{amsart}
\usepackage{amssymb}
\usepackage{amsmath}
\usepackage{hyperref}
\usepackage{bbm}
\usepackage{xcolor}
\definecolor{refkey}{gray}{.5}   
\definecolor{labelkey}{gray}{.5} 
\usepackage{graphicx}
\usepackage{tikz-cd}
\usepackage{epsfig}
\usepackage{mathtools,amsfonts}

\newtheorem{thm}{Theorem}[section]
\newtheorem{prop}{Proposition} [section]
\newtheorem{lem}{Lemma} [section] 
\newtheorem{cor}{Corollary}[section]

\theoremstyle{definition}
\newtheorem{defi}{Definition}

\newtheorem{remark}{Remark}

\numberwithin{equation}{section}
\newcommand{\N}{\mathbb{N}}  

\DeclareMathOperator{\lip}{Lip}
\DeclareMathOperator{\essinf}{ess\,inf}
\DeclareMathOperator{\esssup}{ess\,sup}

\usepackage{xcolor}
\usepackage{soul}

\begin{document}
	
	\title[On $\mu$-Dvoretzky random covering of the circle.]{On $\mu$-Dvoretzky random covering of the circle.}
	\author{AIHUA FAN}	\author{Davit Karagulyan}
	\address{Aihua Fan: LAMFA, UMR 7352, CNRS, University of Picardie, 33 rue Saint Leu, 80039 Amiens CEDEX
		1, France}
	\email{\href{mailto:ai-hua.fan@u-picardie.fr}{ai-hua.fan@u-picardie.fr}}

	\address{Davit Karagulyan: Department of Mathematics, University of Maryland, College Park 20742, USA}
	\email{dkaragul@umd.edu}


	\begin{abstract}
In this paper we study the Dvoretzky covering problem with non-uniformly distributed centers. When the probability law of the centers admits an absolutely continuous density which satisfies 
a regular condition on the set of essential infimum points,  we give a necessary and sufficient condition for covering the circle. When the lengths of covering intervals are of the form $\ell_n = \frac{c}{n}$,  we give a necessary and sufficient condition for covering the circle, without imposing any regularity on the density function.
	\end{abstract}
	\maketitle
	

	
	
	\section{Main Statement}

	Let $\mu$ be a Borel probability measure on the circle $\mathbb{T}
	:= \mathbb{R}/\mathbb{Z}\equiv [0,1)$ identified with the interval $[0,1)$ and let $(\ell_n)_{n\ge 1}$ be a sequence of
	positive numbers  with $0<\ell_n <1$. 
	 Assume that $(\xi)_{n\ge 1}$ is  a sequence of i.i.d. random variables  having $\mu$ as probability law. Then, for each $n\ge 1$, we consider the random interval
	 $$I_n : = (\xi_n -\ell_n/2, \xi_n +\ell_n/2)$$ 
	  of length $\ell_n$ and centered at $\xi_n$. Sometimes, we say that $I_n$
	is the ball $B(\xi_n, r_n)$ centered at $\xi_n$ and of radius
	$r_n := \ell_n/2$. Under what condition on $\mu$ and on $(\ell_n)$ have we
	$$
	    P\big( \mathbb{T}= \limsup_{n\to \infty} I_n\big) =1 ?
	$$
	If the answer is affirmative, we say that $\mathbb{T}$ is covered 
	for the {\em $\mu$-Dvoretzky covering}. 
		We can also ask if a given compact set is covered or not. 
		Without loss of generality, we always assume that $(\ell_n)$ is decreasing.
	 
	\medskip
	 
	When $\mu$ is the Lebesgue measure on $\mathbb{T}$, it is the
	classical Dvoretzky covering problem \cite{Dvoretzky1956}(1956), to which a necessary and sufficient condition for $\mathbb{T}$ to be covered is 
	\begin{equation}\label{cond:Shepp}
	    \sum_{n=1}^\infty \frac{1}{n^2}\exp(\ell_1+\cdots +\ell_n)=\infty.
	\end{equation}
	This is called {\em Shepp's condition}, which was obtained by L. Shepp \cite{Sp}(1972). A compact set $F \subset \mathbb{T}$ is covered if and only if 
	\begin{equation}\label{cond:Kahane}
	{\rm Cap}_{\Phi}(F)=0
	\end{equation}
	where
	$$
	        \Phi(t, s) = \exp \sum_{n=1}^\infty (\ell_n -|t-s|)_+.
	$$
	This is called {\em Kahane's condition}, which was obtained by J-P. Kahane \cite{Kahane1998}(1987). Recall that ${\rm Cap}_{\Phi}(F)=0$ means that
	for any Borel probability measure $\sigma$ supported by $F$ we have
	$$
	    \int_{\mathbb{T}} \int_{\mathbb{T}} \Phi(t,s) d\sigma(t) d\sigma(s) =+\infty.
	$$
	We refer to \cite{19kahane1985some} for the theory of capacity. 
	The above cited results due to Shepp and Kahane 
	will be our basic useful facts.
	Kahane's book \cite{Kahane1985} contains references on the study of classical Dvoretzky covering before 1985. For later works, let us only cite \cite{BarralFan,
		Durand2010, EP2018, Fan1995, Fan2002, Fan2004,  FK1993, FK2001,FLL2010, FW2004, FJJS2018, JJKLSX2017, JS2008, OSW2017,Seuret2018,Yu2003}. A first study on  $\mu$-Dvoretzky covering problem for Gibbs measures $\mu$ was  made in \cite{tang}, where the author used the method taken from \cite{A-Sh-T} in order to find the optimal covering exponent $t$ for $\ell_n = a/n^t, (a>0, t>0)$. 
	
\medskip
	
	In this paper, we will study the $\mu$-Dvoretzky covering problem when $\mu$ admits a density. In the following, $f$ will always denote a density function and the measure of density $f$ will be
	denoted by $\mu_f$.  The solution to the problem depends on the 
	essential infimum of the density function. 
	\begin{defi}
		For a given Borel function $f:\mathbb{T} \rightarrow \mathbb{R}$, the {\em essential infimum} of $f$ is  denoted by $ m_f$, i.e.
		$$
		m_f:= \essinf_{\mathbb{T}}f=\sup \{a \in \mathbb{R}: a\le f(x) \ a.e. \}
		$$
		where "a.e." refers to the Lebesgue measure.
	\end{defi}
We can define the local essential infima function $\underline{E}_f(x)$ and the local essential supprema function $\overline{E}_f(x)$
(see  Section \ref{sect:essential-inf} for definitions). 
 The set of essential infimum points of $f$, denoted by $K_f$ is defined to be the set of those $x$ such that  $\underline{E}_f(x) = m_f$.  It will be proved that $K_f$ is a non-empty compact set (Proposition \ref{prop:EI}). 
 \medskip
 
 The flatness of the density will also play a role. It is considered as a regularity of the density.

\begin{defi}\label{flat}
A point $x \in \mathbb{T}$ has the 
{\em flatness property} for the measure $\mu_f$ and the sequence $\{\ell_n\}_{n \geq 1}$ if 
\begin{equation}\label{bound1}
\sum_{n=1}^\infty |\mu_f(B(x, r_n))-m_f \ell_n|<\infty.
\end{equation}
The set of all flat points will be denoted by $F_f(\{\ell_n\}_{n \geq 1})$ or simply by $F_f$. 
\end{defi}

	Here is another regularity of the density function $f$: there exists a sequence of points $\{x_n\}_{n \geq 1}\subset \mathbb{T}$ such that
	\begin{equation}\label{dom_1}
	\lim_{n \rightarrow \infty}\overline{E}_f(x_n)=m_f.
	\end{equation}
In is easy to see, that if $f$ is continuous at least at one point of the set $K_f$, then the condition \eqref{dom_1} is fulfilled.

We are now ready to state one of the main theorems in this paper.
\begin{thm}\label{thm1}
Let $\mu_f$ be a Borel probability measure on $\mathbb{T}$ with the density function $f$ and $\{\ell_n\}_{n=1}^\infty$ a sequence of positive numbers. Assume that the condition \eqref{dom_1} is fulfilled. 
We distinguish two cases.

$(1)$ Assume
$m_f = 0$ and $K_f$ is countable. Then, the circle is covered for the $\mu_f$-Dvoretzky covering if and only if the following two conditions are satisfied:
\begin{equation}\label{cond0}
\forall x \in K_f,\quad \sum_{n=1}^\infty \mu_f(B(x, r_n))=\infty;
\end{equation}

\begin{equation}\label{cond01}
\forall a> 0, \sum_{n=1}^\infty \frac{1}{n^2}e^{a(\ell_1+\dots + \ell_n)}=\infty.
\end{equation}

$(2)$ Assume $m_f>0$ and there exists a sequence $\{a_n\}_{n \geq 0}$, with $a_0 =0$, so that the set $K_f \setminus \cup_{n \geq 0}(a_n + F_f \cap K_f)$ is at most countable. Then, the circle is covered for the $\mu_f$-Dvoretzky covering if and only if the following two conditions are satisfied:
\begin{equation}\label{cond1}
\forall a> m_f, \sum_{n=1}^\infty \frac{1}{n^2}e^{a(\ell_1+\dots + \ell_n)}=\infty; \quad {\rm Cap}_{\Phi^{(m_f)}}(F_f \cap K_f)=0,
\end{equation}
where the capacity refers to the kernel
$$
      \Phi^{(m_f)}(t,s) = \exp \sum_{n=1}^\infty m_f (\ell_n - |t-s|)_+.
$$ 
\end{thm}

We have the following corollaries of Theorem \ref{thm1}.
\begin{cor}\label{cor1}
	Assume that for all $x \in K_f$, there is an open neighborhood  $U$ of $x$ and a Lipschitz function $g_x \in \lip(U)$, such that $f(t) \leq g_x(t)$ for almost every $t \in U$ and $f(t)=g_x(t)$ on $K_f \cap U$. \\
	\indent {\rm (1)}  If $m_f=0$ and $\sum_{n=1}^\infty \ell_n^2<\infty$,  there will be no $\mu_f$-Dvoretzky covering. \\
		\indent {\rm (2)}
	If $m_f>0$,  the circle is covered if and only if the following two conditions are satisfied
\begin{equation}
\forall a> m_f, \sum_{n=1}^\infty \frac{1}{n^2}e^{a(\ell_1+\dots + \ell_n)}=\infty; \quad {\rm Cap}_{\Phi^{(m_f)}}(K_f)=0,
\end{equation}

\begin{cor}\label{cor2}
Assume there exists $U \subset \mathbb{T}$, so that $f(x)=m_f$, for a.e. $x \in U$. Then the circle is covered for the $\mu_f$-Dvoretzky covering if and only if
$$
\sum_{n=1}^\infty \frac{1}{n^2}e^{m_f(\ell_1+\dots + \ell_n)}=\infty.
$$
\end{cor}

\end{cor}

In the special case of $\ell_n = \frac{c}{n}$ ($c>0$), we don't require
any regularity of $f$.
\begin{thm}\label{thm_seq}
Suppose $\ell_n = \frac{c}{n}$ ($c>0$). Let $\mu_f$ be an arbitrary 
Borel probability measure with density $f$.  A necessary and sufficient condition for covering the circle in the $\mu_f$-Dvoretzky covering is
$
cm_f\geq 1.
$
\end{thm}

The $\mu$-Dvoretzky covering problem is subtle
and is not treated in this paper when $\mu$ is singular. But we have the following 
local comparison principle. This principle will serve us as a tool 
in our present study and it has its own interests.

\begin{thm}\label{thm3} Consider two Dvoretzky covering respectively defined by two Borel probability measures $\mu$ and $\nu$. Assume that
	$\mu|_U \le \nu|_U$ for some non-empty open set $U \subset \mathbb{T}$. Let $K\subset U$ be a compact set in $U$. If $K$    is covered for the $\mu$-Dvoretzky covering, then it is covered for the $\nu$-Dvoretzky covering.
\end{thm}

We organize the rest of the paper as follows. In Section 2, we prove that the set of essential infimum is a non-empty compact set. Section 3 contains two basic results which are respectively qualified local Billard criterion and local Kahane criterion. In Section 4, we prove the comparison principle (Theorem \ref{thm3}) which is our third basic result. After these preparations,  we  find sufficient conditions for covering $\mathbb{T}$ in Section 5, and necessary conditions for covering $\mathbb{T}$ in Section 6. Then Theorem \ref{thm1} together with its corollaries and Theorem \ref{thm_seq}
are proved in Section 7. 

\section{Set of essential infimum}\label{sect:essential-inf}
 We leave the Dvoretzky covering problem for a while. 
In this section, we study the set of points where a measurable function attains its "minimal" value.
\medskip

Let $f:\mathbb{T} \rightarrow \mathbb{R}$ be a Borel measurable function and $I \subset \mathbb{T}$ an interval. 
The {\em essential infimum} of the function $f$ on the interval $I$ is defined as follows
$$
\essinf_{I}f := \sup\{a \in \mathbb{R}: a \le f(x) \ {\rm for\ almost \ all } \ x\in I\}.
$$ 
We will denote $\essinf_{\mathbb{T}}f$ by $m_f$.
Let $x_0 \in \mathbb{T}$ be fixed. The {\em  essential infimum} at  $x_0$ of $f$ is defined  to be the following limit
$$
\underline{E}_f(x_0) : = \lim_{n \rightarrow \infty}\essinf_{B(x_0,\frac{1}{n})}f.
$$
The {\em set of essential infimum points},
denoted $K_f$, is defined by
$$
K_f	:=\{x \in \mathbb{T}: \underline{E}_f(x) =m_f\}.
$$

Similarly  we  define the essential suprema $\esssup_{I}f$
and $\overline{E}_f(x_0)$.
Clearly
$$
\underline{E}_f(x_0) \leq \overline{E} _f(x_0).
$$
It is also cleat that $\underline{E}_f(x_0) = \overline{E} _f(x_0)$ if $f$ is continuous at $x_0$. 

\begin{prop}\label{prop:EI}
For every Borel measurable function $f: \mathbb{T} \to \mathbb{R}$, $K_f$ is a non-empty compact set.
\end{prop}
\begin{proof}
We prove the non-emptyness of $K_f$  by a dissection argument using the fact 
$$\essinf_{I \cup J} f = \min (\essinf_I f, \essinf_J f).
$$
Indeed, cut $\mathbb{T}$ into $\mathbb{T} = I \cup J = [0,1/2) \cup [1/2, 1)$. Then we have 
$$
\essinf_{\mathbb{T}} f = \min(\essinf_I f,\essinf_J f).
$$ 
We continue this process and construct a sequence of nested intervals 
$$
I_1 \supset I_1 \supset \dots \ I_n \supset \dots,
$$
so that $I_{n+1}$ is one of the two halves of the interval $I_n$ and such that
$$
\essinf_{I_n} f = \essinf_{I_{n+1}} f=\essinf_{\mathbb{T}} f.
$$   
Since $|I_{n+1}|=|I_n|/2$, then $\bigcap_{n=1}^\infty I_n$ is a single point, say $\{x_0\}$. 
We claim that $x_0 \in K$. Indeed, for arbitrary $n \in \mathbb{N}$ there is $m \in \mathbb{N}$ such that
$$
B(x_0, 1/n)\supset I_m,
$$
which implies
$$
\essinf_{B(x_0, \frac{1}{n})}f\leq  \essinf_{I_m}f=\essinf_{\mathbb{T}}f.
$$
Let $n\to \infty$. Thus we get $\underline{E}_f(x_0) = \essinf_{\mathbb{T}}f$. 

Now show that $K_f$ is closed. Assume $\{x_n\}_{n \geq 1}\subset K_f$ and $\lim_{n \rightarrow \infty}x_n = x_0$. 
For any $n \in \mathbb{N}$ there exists $m \geq 1$ such that
$$
x_m \in B(x_0, \frac{1}{n}),
$$
which implies that
$$
\essinf_{{B(x_0,\frac{1}{n})}}f\leq \underline{E}_f(x_m)= m_f.
$$
It follows that 
$
\underline{E}_f(x_0)=m_f
$, i.e. $x_0\in K_f$.
\end{proof}

We now show that the regularity condition (\ref{dom_1}) is not always fulfilled.

\medskip
Let $A \subset [0,1]$ be a set, which is constructed in the same way as the classical Cantor set, but at each step, we remove the central interval of shorter length. This results in a compact set with positive Lebesgue measure, which  contains no intervals and is nowhere dense. Consider the characteristic function of the complementary of $A$:
$$
f_A(x) = 1 - \chi_A(x), \quad x \in [0,1).
$$
Observe that
\begin{equation}\label{Cantor}
\forall x \in [0,1],\ \overline{E}_{f_A}(x) = 1;  \quad
\forall x \in A, \ \underline{E}_{f_A}(x) = 0.
\end{equation}
Then, we define the following function
$$
   f(0)= 0; \quad  \forall x \in  [2^{-n-1}, 2^{-n}), 
f(x) = 
f_A(2^{n+1}x - 1)+\frac{1}{n+1}  \ \ \ (n\ge 0).
$$
From (\ref{Cantor}), we get that for each $n\ge 0$ we have
$$
\essinf_{[\frac{1}{2^{n+1}}, \frac{1}{2^n}]}f=\frac{1}{n+1}; \quad 
\forall x \in  [2^{-n-1}, 2^{-n}),
\overline{E}_{f}(x)= 1 + \frac{1}{n+1}.
$$
Finally we get 
$$
K_f=\{0\}; \quad
\underline{E}_{f}(0) = 0;  \quad \forall x \in [0,1), \overline{E}_{f}(x) \geq 1.
$$
Thus, the Condition \ref{dom_1} is not fulfilled by this function $f$. 
However,  recall that the condition 
(\ref{dom_1}) is satisfied for a function $f$ which is continuous at one point of $K_f$.
\medskip

\section{Two basic results due to Billard and Kahane}

The proof of Theorem \ref{thm1} will be based on Theorem \ref{thm3} and on the following two criteria, which have their own interests.

\subsection{Local Billard necessary condition}
The following 
theorem gives us a necessary condition for covering the circle.
The idea of  second moment used in the proof  came from P. Billard \cite{Billard1965}. 
Therefore the condition will be refered to as (local) Billard condition.

	\begin{thm}[local Billard criterion]\label{thm:B} 
		  	Let $F \subset \mathbb{T}$ be a non-empty compact set. 
		  	Suppose that $\mu$ is a probability measure on $\mathbb{T}$ such that
		  	\begin{equation}\label{eq:cond-noncover1}
		  	       \sup_{t \in F} \sum_{n=1}^\infty \mu(B(t, r_n))^2 <\infty.
		  	\end{equation}
		  	Then $F$ is not covered for the $\mu$-Dvoretzky covering if there exists a probability measure	$\sigma$ supported by $F$ such that	
		  	\begin{equation}
		  	       \label{eq:cond-noncover2}
 \int_{F} \int_{F} \exp\sum_{n=1}^\infty \mu(B(t, r_n)\cap B(s,r_n))d\sigma(t)d\sigma(s) <\infty.
		  	\end{equation}
	\end{thm}

\begin{proof} 
	 Consider the martingale
$$
M_n = \int_{F} Q_n(t) d\sigma(t)
$$
where
$$
Q_n(t) = \prod_{j=1}^n \frac{1- \chi_{B(\xi_n, r_n)}(t)}{1 - \mu(B(t, r_n))}.
$$
Notice that $Q_n(t) =0$ means that $t$ is covered by one of the intervals $I_j$ ($1\le j \le n$). If $M_n$ doesn't tend to zero, then some point in $F$ is not covered. 
By Fubini theorem and simple computations, we get
$$
\mathbb{E} M_n^2 = 
\int_{F}  \int_{F}  \prod_{j=1}^n \frac{1-\mu(B(t, r_n) - \mu(B(s, r_n)) + \mu (B(t, r_n) \cap B(s, r_n))}{(1-\mu(B(t, r_n))(1- \mu(B(s, r_n)))}dtds.
$$
Using $1-x = e^{-x + O(x^2)}$ and the conditions (\ref{eq:cond-noncover1}) and (\ref{eq:cond-noncover2}), we get $\mathbb{E}M_n^2 = O(1)$ which implies that the limit of $M_n$ is not almost surely zero.    
\end{proof}

The condition (\ref{eq:cond-noncover2}) means that $F$ has a positive capacity. 
The condition (\ref{eq:cond-noncover1}) can be relaxed to the pointwise finiteness, because $F$ can be approximated by
compact sets on each of which the sum is uniformly bounded.

\subsection{Local Kahane criterion}
The next result can be considered as a local version of Kahane's theorem. For $a>0$, define the kernal
$$
     \Phi^{(a)}(t, s) := \exp{\left( a \sum_{k=1}^\infty (l_k - |t-s|)_+\right)}.
$$
As we see below, this criterion is not perfect because we need the assumption that the density is constant around the set to be covered  in question. But it will be one of our basic tools, because we can approximate our density function by functions which are locally constant.

\begin{thm}[local Kahane criterion]\label{thm:K} \label{kh-loc}Consider a probability measure $\mu_f$
	having its density function $f$. 
Let $I \subset \mathbb{T}$ be an open interval. Suppose that $f(x)=a$ for almost all $x \in I$. Then 
a compact subset  $F$ of $I$ is covered for the $\mu_f$-Dvoretzky covering if and only if ${\rm Cap}_{\Phi^{(a)}}(F) =0$.
\end{thm}
\begin{proof} The idea of proof is to compare the $\mu_f$-Dvoretzky
	covering with a classic Dvoretzky covering. 	
Let $M$ be the distribution function of $\mu_f$, i.e.
$$
     M(x) = \mu_f([0,x]) = \int_0^{x} f(t) dt \quad (0\le x\le 1).
$$
Let $X$ be a random variable which is uniformly distributed on $[0,1)$. Define
$$
           Y= M^{-1}(X)
$$ 
where the inverse function is defined by
$$
      M^{-1}(x) = \inf \{t\in [0,1]: M(t) \ge x\}.
$$
It is well known that the probability law of $Y$ is $\mu_f$. 
So, take a sequence $(\omega_n)$ of i.i.d. random variables uniformly distributed on $[0,1)$ and use $(\xi_n)$ to model a $\mu_f$-Dvoretzky
covering, where $\xi_n = M^{-1}(\omega_n)$.
 
 Observe that the restriction $M: I \to M(I)$ is affine and invertible. It follows that  for any interval  $J \subseteq I$,  its image $M(J)$ is an interval of length 
$$|M(J)|=\int_J f(t)dt=|J|a$$ 
and the center of $M(J)$ is  the image of the center of $J$. 
So,  if $I_n = (\xi_n-\ell_n/2, \xi_n +\ell_n/2) \subset I$, then   $M(I_n)$ is centered at $(\omega_n)$ and of length  $a\ell_n$. 

Now choose a proper subset $J \subset I$ such that $F \subset J$. Then choose $N \in \mathbb{N}$ so large that $\ell_n \leq {\rm dist}\{J, \partial I\}$ for all $n \geq N$.  
Assume $F \subset \limsup I_n(\xi_n)$. Then for any $x\in F$,
$x\in I_n(\xi_n)$ for an infinite number of $n's$ with $n\ge N$ (these $n$'s depend on $x$). These $\xi_n$'s must fall into the interval $J$.   For such $n$, $M(x) \in M_n(I_n(\xi_n))$. 
It follows that 
$$
P(F \subset \limsup_{n \rightarrow \infty} I_n)=1
\Longrightarrow
P(M(F) \subset \limsup_{n \rightarrow \infty} M(I_n))=1.
$$
The converse implication can be similarly proved, because $M: I \to M(I)$
is invertible. 
In other words,  $F \subset J$ is covered for the $\mu_f$-Dvoretzky coverring   by intervals of length $\{\ell_n\}_{n \geq 1}$  if and only if $M(F)$ is covered for the classic Dvoretzky covering by intervals of length $\{a \ell_n\}_{n \geq 1}$.  But, by Kahane's theorem, the set $M(F)$ is covered for the classic Dvoretzky covering if and only if 
${\rm Cap}_{\Phi^{(a)}}(M(F)) =0$, which is equivalent to 
${\rm Cap}_{\Phi^{(a)}}(F) =0$, because $M$ is affine around $F$.
\end{proof}
	
\subsection{Two elementary facts}
Recall that we always assume that $(\ell_n)$ is decreasing.
The following fact is known when $F=\mathbb{T}$. The general case is not trivial.
The proof given below involves both Shepp's condition and Kahane's condition. 
	
	\begin{lem}\label{lem-1}
Let $F$ be a compact set with $|F|>0$ and $a>0$ a positive number.
Then the following two conditions are equivalent:
\begin{equation}\label{subs}
{\rm Cap}_{\Phi^{(a)}}(F)=0,
\end{equation}
\begin{equation}\label{whole}
\sum_{n=1}^\infty \frac{1}{n^2} \exp(a(\ell_1+\cdots +\ell_n))= \infty.
\end{equation}
\end{lem}
\begin{proof}
A simple calculation shows that there exists a constant $C_a>0$
such that
$$
\int_F \int_F \exp{\left( a \sum_{k=1}^\infty (\ell_k - |t-s|)_+\right)}ds dt 
\leq C_a
\int_\mathbb{T} \int_\mathbb{T} \exp{\left( \sum_{k=1}^\infty (a \ell_k - |u-v|)_+\right)}du dv.
$$
(First replace $F$ by $\mathbb{T}$ and then make a change of variable). Then (\ref{whole}) is implied by (\ref{subs}) because the last double integral equals to the infinity if and only if 
(\ref{whole}) holds.

Now assume (\ref{whole}). 
It is nothing but Shepp's condition for the classic 
Dvoretzky covering with the sequence of lengths $\{a\ell_n\}_{n \geq 1}$. Then any non-empty compact set $K' \subset \mathbb{T}$
is covered.  If $F'$ has positive Lebesgue measure, it supports  the restriction of Lebesgue measure. Thus, by Kahane's condition, we have
\begin{equation*}
\int_{F'} \int_{F'} \exp{\left( \sum_{k=1}^\infty (a \ell_k - |u-v|)_+\right)}du dv=\infty.
\end{equation*}
Take $F'=aF$, where $aF$ is the scaling of $F$ by a coefficient $a$. 
Then, by making a change of variable,  we get 
\begin{equation*}
\int_{F} \int_{F} \exp{\left( \sum_{k=1}^\infty a( \ell_k - |t-s|)_+\right)}dt ds=\infty. 
\end{equation*}

\end{proof}

The following fact shows that $\sum_{n=1}^\infty \ell_n^2 =\infty$ is a strong condition for the Dvoretzky covering problem.

\begin{lem}\label{2-summable} 
	The condition
	$
	\sum_{n=1}^\infty \ell_k^2 = \infty
	$
	implies 
	$$
	\forall a>0, \ \ \ \sum_{n=1}^\infty \frac{1}{n^2} \exp(a(\ell_1+\cdots +\ell_n))= \infty.
	$$
\end{lem}

\begin{proof}
This is known, however we will include a proof for completeness.	The condition
	$
	\sum_{n=1}^\infty \ell_n^2=\infty
	$
	implies that 	$
	\ell_n > \frac{1}{n^{2/3}}
	$ holds for infinitely many $n$. Then, for such a $n$, by the monotonicity of $(\ell_n)$ we have
	$$
	\sum_{k=1}^n \ell_k > \frac{n}{n^{2/3}}>n^{1/3}.
	$$
	Therefore the general term of the series in question doesn't tend to zero. So, the series diverges. 
\end{proof}

	 \section{A comparison principle: Proof of  Theorem \ref{thm3}}
	 
	 Intuitively, if $\mu|_U\le \nu|_U$, then sets contained in an open set
	 $U$ are easierly covered for the $\nu$-Dvoretzky covering than 
	 for the $\mu$-Dvoretzky covering, because there are more chances
	 to get points in $U$ for the $\nu$-Dvoretzky model.
	 We give here a rigourous proof of this intuition which is stated
	 as Theorem \ref{thm3}. The proof benefits from our discussion with Meng Wu.   
	 
	 \subsection{Model of $\mu$-Dvoretzky covering}
	 
Let $\mu=\mu_0 + \mu_1$ be a decomposition of a probability measure $\mu$ on $\mathbb{T}$. Then we set
$$
\widetilde{\mu}_0 := \frac{\mu_0}{\mu_0(\mathbb{T})}, \quad 
\widetilde{\mu}_1 := \frac{\mu_1}{\mu_1(\mathbb{T})}.
$$
For simplicity, we write $\alpha_0= \mu_0(\mathbb{T})$ and $\alpha_1= \mu_1(\mathbb{T})$, so that $\mu = \alpha_0  \widetilde{\mu}_0 
+  \alpha_1  \widetilde{\mu}_1$.

Let $\Omega= \{0,1\}^\mathbb{N} \times \mathbb{T}^\mathbb{N}$. A point in $\Omega$  is denoted by
$$
\omega := (\epsilon; \xi):= (\epsilon_1, \cdots, \epsilon_n, \cdots; \xi_1,\cdots, \xi_n, \cdots) \in \Omega.
$$
Using the decomposition $\mu = \alpha_0  \widetilde{\mu}_0 
+  \alpha_1  \widetilde{\mu}_1$, we define a Borel probability measure $P$ on $\Omega$ as follows. 
For any integer $n\ge 1$, any $(a_1, \cdots, a_n) \in \{0, 1\}^n$
and any $(A_1, \cdots, A_n) \in \mathcal{B}(\mathbb{T})^n$, define
\begin{equation}\label{def:P}
P(\epsilon_1=a_1, \cdots, \epsilon_n =a_n; \xi_1\in A_1, \cdots, \xi_n \in A_n)
:= \prod_{j=1}^n \alpha_{a_j} \widetilde{\mu}_{a_j}(A_j)
= \prod_{j=1}^n \mu_{a_j}(A_j).
\end{equation}
Notice that $P$ is not exactly a tensor product. But it has a nice 
desintegration as we show below.

Let $\pi$ be the 
$(\alpha_0, \alpha_1)$-Bernoulli measure on  $\{0,1\}^\mathbb{N}$.
For $\epsilon$ fixed, 
let
$$
\Lambda^{(\epsilon)} =\{n \in \mathbb{N}: \epsilon_n =1\},
$$  
and we define two measures respectively 
on $\mathbb{T}^{\Lambda^{(\epsilon)}}$ and  $\mathbb{T}^{\mathbb{N}\setminus \Lambda^{(\epsilon)}}$ where
$$
P_{\Lambda^{(\epsilon)}} = \bigotimes_{n \in \Lambda^{(\epsilon)}} \widetilde{\mu}_{1}, 
\quad 
P_{\mathbb{N}\setminus\Lambda^{(\epsilon)}} = \bigotimes_{n \in \mathbb{N}\setminus \Lambda^{(\epsilon)}} \widetilde{\mu}_{0},
$$
Let $P^{(\epsilon)}=  P_{\Lambda^{(\epsilon)}} \otimes P_{\mathbb{N}\setminus\Lambda^{(\epsilon)}}$. Notice that 
$P^{(\epsilon)}$ is a version of the conditional probability
$P(\cdot | \epsilon)$.

\begin{lem}\label{lem:Lem1} The measure $P$ has the following desintegration
	\begin{equation}\label{eq:desint}
	P = 
	\int P^{\epsilon} d\pi(\epsilon) =\int P_{\Lambda^{(\epsilon)}} \otimes 
	P_{\mathbb{N}\setminus\Lambda^{(\epsilon)}} d \pi(\epsilon).
	\end{equation}
	Consequently, the following statements hold:\\
	\indent {\rm (1)} All the random vectors $(\epsilon_j, \xi_j)$ ($j =1,2,\cdots$) are i.i.d..\\
	\indent {\rm (2)} Each $\epsilon_j$ obeys the $(\alpha_0, \alpha_1)$-Bernoulli law;\\
	\indent {\rm (3)} Each $\xi_j$ admits $\mu$ as probability law;\\
	\indent {\rm (4)} 
	For  $\epsilon$ fixed, with respect to $P^{(\epsilon)}$, $\{\xi_n\}_{n\in \Lambda^{(\epsilon)}}$ are i.i.d. random variables with probability law $\widetilde{\mu}_1$ and $\{\xi_n\}_{n\in \N\setminus\Lambda^{(\epsilon)}}$ are i.i.d. random variables  with probability law $\widetilde{\mu}_0$.
\end{lem}
\begin{proof}
	Using the fact
	$
	{\mu}_{0}+  {\mu}_{1}
	= \mu
	$, we check that the measure $P$ is well defined.  The desintegration (\ref{eq:desint}) is nothing but the definition (\ref{def:P}). It follows the other properties.
\end{proof}

Since $(\xi_n)$ is a sequence of i.i.d. random variables with $\mu$
as probability law, the random intervals
$I_n(\xi_n) = (\xi_n - \ell_n/2, \xi_n + \ell_n/2)$  define a model of $\mu$-covering. Recall that $\mu = \mu_0 + \mu_1$

\begin{lem}\label{lemma comparison1}
	Let $U \subset \mathbb{T}$ be a non-empty open interval.
	Suppose ${\mu}_0(U)=0$. Then for any compact set  $K\subset U$, for $P$-a.e. $(\epsilon,\xi)\in \Omega$, we have
	\begin{equation}\label{eq. lem comp1}
	K\subset \limsup_{n\in \mathbb{N}} I_n(\xi_n) \Longleftrightarrow  K \subset \limsup_{n\in \Lambda^{(\epsilon)}} I_n(\xi_n).
	\end{equation}
\end{lem}	
\begin{proof}
	Recall that conditioned on $\epsilon\in \{0,1\}^\N$, the sequence $\{\xi_n\}_{n\in \Lambda^{(\epsilon)}}$ are i.i.d. random variables with probability law $\widetilde{\mu}_1$ and $\{\xi_n\}_{n\in \N\setminus\Lambda^{(\epsilon)}}$ are i.i.d.  with probability law $\widetilde{\mu}_0$. Since $\widetilde{\mu}_0(U^c)=1$,  for $P^{(\epsilon)}$-a.e. $\xi$ we have $\xi_n\in U^c$ for all $n\in \N\setminus\Lambda^{(\epsilon)}$. Since $K$ is compact, it has positive distance from the boundary of $U$, we must have $$
	P^{(\epsilon)}(K\cap \limsup_{n\in \N\setminus\Lambda^{(\epsilon)}} I_n(\xi_n)=\emptyset)=0.
	$$ From this we deduce \eqref{eq. lem comp1}, with the aid of the desintegration (\ref{eq:desint}).
\end{proof}

However, notice that $\epsilon_j$ and $x_j$ are not independent.

\subsection{Proof of Theorem \ref{thm3}}

	We can model the covering as above, according to the following decompositions
	\begin{eqnarray*}
		\mu &=& \mu_0+\mu_1 \ \ \ {\rm with}\ \ \   \mu_1=\mu|_U, \mu_0=\mu|_{U^c} \\
		\nu &=& \nu_0+\nu_1 \ \ \ \ {\rm with}\ \ \  \nu_1= \mu|_U,  \nu_0 = (\nu|_U-\mu|_U)+ \nu|_{U^c}.
	\end{eqnarray*}
	The corresponding measures will be denoted respectively by $P_\mu$ and $P_\nu$.
	
	Suppose $K$ is a.s. covered in the $\mu$-covering. Then by Lemma \ref{lemma comparison1}, we have
	$$
	\pi\!-\!a.e.\ \epsilon, \ \ \ P_\mu^{(\epsilon)}\left(K \subset \limsup_{n\in \Lambda^{(\epsilon)}} 
	I_n(\xi_n)\right)=1.
	$$
	But $(\xi_n)_{n \in \Lambda^{(\epsilon)}}$ has the same probability law under
	$P_\mu^{(\epsilon)}$ and under $P_\nu^{(\epsilon)}$, because of $\mu_1 = \nu_1$ (see Lemma \ref{lem:Lem1} (4)). Thus
	$$
	\pi\!-\!a.e.\ \epsilon, \ \ \ P_\nu^{(\epsilon)}\left(K \subset \limsup_{n\in \Lambda^{(\epsilon)}} 
	I_n(\xi_n)\right)=1.
	$$
	It follows obviously that 
	$$
	\pi\!-\!a.e.\ \epsilon, \ \ \ P_\nu^{(\epsilon)}\left(K \subset \limsup_{n\in \mathbb{N}} 
	I_n(\xi_n)\right)=1.
	$$ 
	Finally using once more the desintegration we get
	\begin{eqnarray*}
		P_\nu\left(K \subset \limsup_{n\in\N} I_n(\xi_n)\right)=
		\int 
		P_\nu^{(\epsilon)}\left(K \subset \limsup_{n\in \mathbb{N}} 
		I_n(\xi_n)\right) 
		d\pi (\epsilon) =1.
	\end{eqnarray*}
	
\subsection{Comparison with the classic Dvoretzky covering}
\begin{cor}\label{f-0}
	Let us consider a probability measure $\mu_f$ on $\mathbb{T}$ with density $f$.
	Let $F$ be a non-empty compact set contained in an open set $U$.  
	
	1) Assume $f(x) \geq a$ for almost all $x \in U$. The set $F$
	is covered for the $\mu_f$-Dvoretzky covering if 
	${\rm Cap}_{\Phi^{(a)}}(F)=0$.

	2)
	Assume $f(x) \leq a$ for almost all $x \in U$. 
	The set $K$
	is not covered for the $\mu_f$-Dvoretzky covering if 
	${\rm Cap}_{\Phi^{(a)}}(F)>0$.
	\begin{proof}
		Let $h$ be a density function on $\mathbb{T}$ so that $h(x)=a$ for all $x \in U$. 
		
		1) Assume 	${\rm Cap}_{\Phi^{(a)}}(F)=0$. Then, according to the local Kahane criterion (Theorem \ref{thm:K}), the set $F$
		is covered for the $\mu_h$-Dvoretzky covering. Then, by Theorem \ref{thm3},  the set $F$
		is covered for the $\mu_f$-Dvoretzky covering for
		$f|_U \ge h|_U$.
		
			2) Assume 	${\rm Cap}_{\Phi^{(a)}}(F)>0$. Then, according to the local Kahane criterion (Theorem \ref{thm:K}), the set $F$
		is not covered for the $\mu_h$-Dvoretzky covering. Then, by Theorem \ref{thm3},  the set $F$
		is not covered for the $\mu_f$-Dvoretzky covering for
		$f|_U \le h|_U$..
	\end{proof}
\end{cor}	
	
\section{Sufficient conditions for covering $\mathbb{T}$}\label{sct-5}
In this section we discuss several sufficient  conditions for covering the circle, which is decomposed into $K_f$ and $\mathbb{T}\setminus K_f$.

\begin{prop}\label{l-1}
The set $\mathbb{T}\setminus K_f$ is covered for the $\mu_f$-Dvoretzky covering under the condition
\begin{equation}\label{eq12}
		    \forall a> m_f \,, \quad 
		    \sum_{n=1}^\infty \frac{1}{n^2} \exp(a(\ell_1+\cdots +\ell_n))= \infty.
\end{equation}
If, additionally, $f$ satisfies the property \eqref{dom_1}, then the condition \eqref{eq12} will also be necessary for covering the set $\mathbb{T}\setminus K_f$. 
\end{prop}

\begin{proof} Let $K_f^c =\mathbb{T}\setminus K_f$.
Let us assume \eqref{eq12}. Notice that  $K_f^c$ is an open set,
 by Proposition \ref{prop:EI}.  By definition,
$
\underline{E}_f(x) > m_f
$ for all $x \in K_f^c$.
Therefore, for every $x \in K_f^c$ there is an integer $n_x \in \mathbb{N}$ and a number $a_x>m_f$ such that
$$
f(t)\geq a_x, \hbox{ for almost all }t \in B(x,\frac{1}{n_x}).
$$
Thus, by Corollary \ref{f-0}, the interval $(x-\frac{1}{n_x}, x+\frac{1}{n_x})$ is covered for the $\mu_f$-Dvoretzky covering. 
However, the open set $K^c_f$ is a union of countably many
such intervals $B(x,\frac{1}{n_x})$. So, $K_f^c$ is covered. 

Now assume  the condition \eqref{dom_1}, which implies that for any 
$\varepsilon>0$, there is $N \in \mathbb{N}$ such that
$
\overline{E}_f(x_N) < m_f + \varepsilon,
$
which implies that there is $k \in \mathbb{N}$ such that
$$
f(x)< m_f + \varepsilon, \hbox{ for almost all }x \in B(x_0, 1/k).
$$
Let $a = m_f + \varepsilon$. 
Notice that $\varepsilon$ is arbitrary. By Corollary \ref{f-0}, the condition \eqref{eq12} must be necessary for covering $K_f^c$.
\end{proof}

\begin{cor}\label{cont}
If $f$ is continuous at some point of $K_f$, then the condition \eqref{eq12} is necessary and sufficient for covering the set $\mathbb{T}\setminus K_f$. 
\end{cor}
\begin{proof}Assume $f$ is continuous at $x_0 \in K_f$. 
	Then $
	\overline{E}_f(x_0)=m_f,
	$
	which clearly implies \eqref{dom_1} if we take $x_n=x_0$
	for all $n\ge 1$.

\end{proof}

\subsection{Sufficient conditions for covering $\mathbb{T}$}

\begin{prop}\label{suff}
Assume $m_f>0$.	The circle is covered for the $\mu_f$-Dvoretzky covering if 
	\begin{equation}\label{suff-cond1}
	\forall a >m_f, \ \ \ \sum_{n=1}^\infty \frac{1}{n^2}e^{a(\ell_1+\dots + \ell_n)}=\infty; 
	\quad {\rm Cap}_{\Phi^{(m_f)}}(K_f)=0.
	\end{equation}
\end{prop}
\begin{proof} The first condition in (\ref{suff})
	implies that $\mathbb{T}\setminus K_f$ is covered,  	Proposition \ref{l-1}. The covering of $K_f$ follows from
	the second condition in (\ref{suff}), Corollary \ref{f-0} and Kahane's result.
\end{proof}


\begin{prop}\label{c-2}
	If we have $m_f>0$, then the condition 
	$$
	\sum_{n=1}^\infty \frac{1}{n^2} \exp(m_f(\ell_1+\cdots +\ell_n))= \infty,
	$$
	is sufficient for covering the circle.
\end{prop}
\begin{proof} It follows from  Corollary \ref{f-0} and Shepp's result, because
	$f(x)\geq m_f$ for almost all $x \in \mathbb{T}$.
\end{proof}

We now show that each of the conditions in \eqref{cond1} does not imply the other.

Let us recall the two conditions in \eqref{cond1} : 
\begin{equation*}
{\rm (C1)} \ \ \ \forall a> m_f, \sum_{n=1}^\infty \frac{1}{n^2}e^{a(\ell_1+\dots + \ell_n)}=\infty; \qquad \qquad
{\rm (C2)} \ \ \ {\rm Cap}_{\Phi^{(m_f)}}(K_f)=0.
\end{equation*}
Consider the following density function
$$
f(x)=|x| + \frac{3}{4}, \hbox{ for } x \in \left[-\frac{1}{2},\frac{1}{2}\right).
$$
Observe that $m_f=\frac{3}{4}$, $K_f=\{0\}$  and $f$ is Lipschitz at $x=0$.  By Corollary \ref{cor1} (2), the conditions (C1) and (C2) are necessary and sufficient for covering the circle. Since $K_f=\{0\}$, the condition (C2) is equivalent to 
$
\sum_{n=1}^\infty \ell_n = \infty,
$
which does not imply (C1). Indeed, when $\ell_n = \frac{3}{4n}$, we have
$
\sum_{n=1}^\infty \ell_n = \infty,
$ but 
$$
\forall a<4/3, \ \ \ \sum_{n=1}^\infty \frac{1}{n^2}e^{a(\ell_1+\dots + \ell_n)}\approx\sum_{n=1}^\infty \frac{1}{n^{(2-\frac{3a}{4})}}<\infty.
$$
Observe however that in this case (C1) implies (C2).

We now show that (C1) does not imply (C2). To see this, let us look at the sequence
$
\ell_n = \frac{2}{n}-\frac{4}{n\ln n}
$
and the density function
$$
f(x)=  
\frac{1}{2} 1_{[0,1/2)}(x) + 
\frac{3}{2} 1_{[1/2, 1)}(x).
$$
We have that $m_f=1/2$ and $K_f = [0,1/2]$. From Theorem \ref{thm_seq}, the conditions (C1) and (C2) are necessary and sufficient for covering the circle.
Note that for all $a>1/2$
$$
\sum_{n=1}^\infty \frac{1}{n^2}e^{a(\ell_1+\dots + \ell_n)}\approx\sum_{n=1}^\infty \frac{1}{n^2}e^{(2a\ln n - 4a \ln \ln n)}=\sum_{n=1}^\infty \frac{1}{n^{(2-2a) } \ln^{4a} n}=\infty.
$$
Hence, (C1) is satisfied. Since $|K_f|>0$, according to Lemma \ref{lem-1} the condition (C2) is equivalent to \eqref{whole}, which reads as
$$
\sum_{n=1}^\infty \frac{1}{n^2}e^{\frac{1}{2}(\ell_1+\dots + \ell_n)}=\infty.
$$
However
$$
\sum_{n=1}^\infty \frac{1}{n \ln^{2} n}<\infty.
$$
Thus, the condition (C1) does not imply (C2) for all density functions $f$ and all sequences $\{\ell_n\}_{n \geq 1}$.
\medskip

\section{Necessary conditions for covering $\mathbb{T}$}

In this section we discuss  necessary conditions for covering the circle. The proof of the next proposition is based on Lebesgue's 
differentiation theorem and Billard's criterion.

\begin{prop}\label{cond_seq}
Suppose that  the sequence $\{\ell_n\}_{n=1}^\infty$ satisfies the condition
\begin{equation}\label{s2}
\limsup_{n \rightarrow \infty}\frac{n \ell_n}{ \ell_1 + \cdots + \ell_n}<1.
\end{equation}
If $\mathbb{T}$ is covered for the $\mu_f$-Dvoretzky covering, then
\begin{equation}\label{suff-cond1-}
\forall a >m_f, \ \ \ \sum_{n=1}^\infty \frac{1}{n^2}e^{a(\ell_1+\dots + \ell_n)}=\infty.
\end{equation}
\end{prop}
\begin{proof} For  $\epsilon>0$, let
	$$
	A = \left\{ x\in \mathbb{T}:f(x)<m_f+ \frac{\epsilon}{2}\right\}.
	$$  
By the definition of $m_f$,   $A$ has positive Lebesgue measure.
Consider 
$$
S_n(x):=\frac{\mu_f(B(x, r_n))}{\ell_n}=\frac{1}{\ell_n}\int_{x-r_n}^{x+r_n}f(t)dt.
$$
By  Lebesgue's differentiation theorem, 
$
S_n(x)$ converges to $f(x)$ almost everywhere. Then,
by Egoroff's theorem,   there is a compact set $A' \subset A$ with $|A'|>0$ such that $
S_n(x)$ converges to $f(x)$ uniformly in $A'$. Therefore, for all large enough $n$, we have
$$
\forall x \in A', \ \ \ |S_n(x)- f(x)|<\frac{\epsilon}{2}.
$$
It follows that for large $n$ and for  $x \in A'$ we  have
\begin{equation}\label{eqn:A}
\mu_f (B(x,r_n))\leq \Big(f(x) + \frac{\epsilon}{2}\Big) \ell_n \leq (m_f + \epsilon)\ell_n.
\end{equation}
Notice that
$
B(t, r_n)\cap B(s,r_n) \subset B(s, r_n)
$
and 
$
f(x)\geq m_f$ for almost all $x \in \mathbb{T}$.
It follows that for all $s, t \in \mathbb{T}$
\begin{equation}\label{eqn:A2}
\mu_f(B(t, r_n)\cap B(s,r_n)) 
\leq m_f(\ell_n - |t-s|)_+ +(\mu
_f(B(s, r_n)-m_f \ell_n).
\end{equation}
Indeed, this is trivial if $|t-s|>\ell_n$. Otherwise the inequality reads as 
$$
   \mu_f(B(t, r_n)\cap B(s,r_n)) \le \mu_f(B(s,r_n))- m_f |t-s|,  
$$
which is  true, because $B(s, r_n)$ is the union of $B(s, r_n)\cap B(t, r_n)$ and an interval of length $|t-s|$.
Assume now $s\in A'$. From (\ref{eqn:A2}) and (\ref{eqn:A}) we get 
\begin{equation} \label{eqn:A3}
\mu(B(t, r_n)\cap B(s,r_n))
\leq m_f(\ell_n - |t-s|)_+ +\epsilon\ell_n.
\end{equation}

We claim that there exists $0<\delta<1$ such that for all 
$(t, s)\in A'\times A'$ with $|t-s|$ small enough we have
\begin{equation} \label{eqn:A4}
 \sum_{n=1}^\infty \mu_f(B(t, r_n)\cap B(s, r_n)
\le \Big(m_f +\frac{\epsilon}{1-\delta}\Big)\sum_{n=1}^\infty (\ell_n -|t-s|)_+.
\end{equation}
Indeed, assume $\ell_{N+1}<|t-s|\le \ell_N$ we have
\begin{equation*}\label{estim}
\sum_{n=1}^\infty \mu_f(B(t, r_n)\cap B(s,r_n)) = \sum_{n=1}^N \mu_f(B(t, r_n)\cap B(s,r_n)).
\end{equation*}
By the estimation (\ref{eqn:A3}), for $(t, s)\in A'\times A'$ we have
\begin{equation}\label{EST}
\sum_{n=1}^\infty \mu_f(B(t, r_n)\cap B(s,r_n)) \leq \sum_{n=1}^N (\ell_n-|t-s|)_+ +\epsilon \sum_{n=1}^N \ell_n.
\end{equation}
However $$
\sum_{n=1}^N (\ell_n-|t-s|)_+= \sum_{n=1}^N \ell_n-N|t-s|\ge 
\sum_{n=1}^N \ell_n-N\ell_N.
$$
By the assumption (\ref{s2}), for some $0<\delta <1$, when $N$ is large enough (i.e. $|t-s|$ is small enough), we have
\begin{equation}\label{inter}
\sum_{n=1}^N (\ell_n-|t-s|)_+ \ge (1- \delta)\sum_{n=1}^N \ell_n. 
\end{equation}
Therefore, by 
(\ref{EST}) and (\ref{inter}), we get
$$
\sum_{n=1}^\infty \mu(B(t, r_n)\cap B(s,r_n))
\leq \left(m_f+\frac{\epsilon}{1-\delta}\right)\sum_{n=1}^N (\ell_n - |t-s|)_+. 
$$
This is what we have claimed for (\ref{eqn:A4}) because 
$(\ell_n -|t-s|)_+=0$ for $n>N$. 

We are ready to check (\ref{suff-cond1-}). By Lemma \ref{2-summable}, 
we can assume that 
$
\sum_{n=1}^\infty \ell_n^2<\infty.
$
Then, by (\ref{eqn:A}), we have
$$
\sup_{t \in A'}\sum_{n=1}^\infty \mu(B(t, r_n))^2 =(m_f + \varepsilon)^2\sum_{n=1}^\infty \ell_n^2<\infty.
$$
Since the compact set $A'$ is covered and it has positive Lebesgue measure, we can apply Billard's local criterion (Theorem \ref{thm:B}) to confirm that
$$
\int_{A'} \int_{A'} \exp\sum_{n=1}^\infty \mu(B(t, r_n)\cap B(s,r_n))dt ds =\infty.  
$$
Then, by the estimation (\ref{eqn:A4}) we get
$$
\int_{A'} \int_{A'} \exp \left((m_f+ \frac{\varepsilon}{1-\delta})\sum_{n=1}^\infty (\ell_k - |t-s|)_+ \right) dt ds  =  \infty.
$$
Notice that $0<\delta<1$ is fixed and $\epsilon >0$ is arbitrary. 
So, we conclude for (\ref{suff-cond1-}) from the last equality and 
Lemma \ref{lem-1}.

\end{proof}

\begin{prop}\label{t-1}
Suppose that $\mu_f$ admits a density $f$.
If  $K_f\cap F_f$ is covered, then ${\rm Cap}_{\Phi^{(m_f)}}(K_f\cap F_f)=0$.
\end{prop}

\begin{proof}
	If $K_f \cap F_f =\emptyset$, there is nothing to prove. Suppose then $K_f \cap F_f\not =\emptyset$. 
Assume  $m_f=0$. Then by  the Borel-Cantelli lemma, any fixed point $x\in K_f$ is only finitely covered. Since $K_f\cap F_f$ is covered, we must have  $m_f>0$.
We can assume that $\sum_{n=1}^\infty \ell_n^2 <\infty$. Otherwise,
we get immediately ${\rm Cap}_{\Phi^{(m_f)}}(K_f)=0$
from Lemma \ref{2-summable} and Corollary \ref{c-2}.

Notice that $x \mapsto \mu_f(B(x, r_n))$ is continuous. So 
the function
$\sum_{n=1}^\infty |\mu_f(B(x, r_n))-m_f \ell_n|$ is lower semi-continuous. Thus  the following sets are compact:
$$
K_m : =K_f\cap \left\{x\in \mathbb{T}: \sum_{n=1}^\infty |\mu_f(B(x, r_n))-m_f \ell_n|\le  m\right\}.
$$
 It is clear that $K_m$ is increasing and it tends to $K_f\cap F_f$. Therefore
 we have only to prove that ${\rm Cap}_{\Phi^{(m_f)}}(K_m)=0$
 for every $m\ge 1$.

By the inequality (\ref{eqn:A2}), for all $t, s \in \mathbb{T}$ we have
$$
\sum_{n=1}^\infty \mu(B(t, r_n)\cap B(s,r_n))\leq \sum_{n=1}^\infty m_f (\ell_k - |t-s|)_+ + \sum_{n=1}^\infty |\mu(B(s, r_n))-m_f \ell_n|.
$$
Assume $s\in K_m$. Then the last sum is bounded by $m$.
Therefore, for all $ s \in K_m$ and all $t\in \mathbb{T}$ we have
\begin{equation}\label{N1}
\sum_{n=1}^\infty \mu(B(t, r_n)\cap B(s,r_n))\leq \sum_{n=1}^\infty m_f (\ell_k - |t-s|)_+ + m.
\end{equation}
On the other hand, letting $a_n=|\mu(B(s, r_n))-m_f \ell_n|$, we have
\begin{equation}\label{N2}
\sup_{s \in K_m}\sum_{n=1}^\infty \mu(B(s, r_n))^2\le  \sup_{s \in K_m} \sum_{n=1}^\infty (a_n^2 + 2 a_n m_f\ell_n +(m_f\ell_n)^2)<\infty 
\end{equation}
because $\sum \ell_n^2 <\infty$ and $\sum a_n \le m$.
By the local Billard criterion, for any probability
measure $\sigma$ on $K_m$ we have 
\begin{equation*} \label{eq1}
\int_{K_m} \int_{K_m} \exp\sum_{n=1}^\infty \mu(B(t, r_n)\cap B(s,r_n))d\sigma(t)d\sigma(s) = +\infty,
\end{equation*}
which, together with (\ref{N1}), implies
\begin{equation}
\int_{K_m} \int_{K_m} \exp\sum_{n=1}^\infty m_f (\ell_k - |t-s|)_+d\sigma(t)d\sigma(s)=\infty.
\end{equation}	
Thus we have proved ${\rm Cap}_{\Phi^{(m_f)}}(K_m)=0$.
\end{proof}

We now discuss the condition \eqref{s2} and  its relation to \eqref{suff-cond1-}.

\medskip

(A) {\em
The condition \eqref{s2} is not always fulfilled.}

 We will construct a decreasing sequence $(\ell_n)$ such that $
\sum_{n=1}^\infty \ell_n^2 < \infty
$ and
\eqref{s2} is not satisfied.
Let $1 = n_0 < n_1 < \cdots < n_k < \dots$ be a sequence of positive integers such that for all $k\geq 1$ 
\begin{equation}\label{d11}
\frac{\ln n_k}{n_k}\geq \frac{\ln n_{k+1}}{n_{k+1}}.
\end{equation}
Define $(\ell_n)$ as follows
$$
 \ell_n := \frac{\ln n_{k+1}}{n_{k+1}}, \ \ {\rm when}\ \ 
n_1 + \dots + n_k <n \leq n_1 + \dots + n_k + n_{k+1}.  
$$
Clearly $(\ell_n)$ is decreasing and if $n_k$ grows fast enough, then
$$
\sum_{n=1}^\infty \ell_n^2 < \sum_{k=1}^\infty n_k \cdot \frac{\ln^2 n_k}{n_k^2}= \sum_{k=1}^\infty\frac{\ln^2 n_k}{n_k}< \infty.
$$
For $n = n_1 + \dots + n_k$, we have
$$
\frac{n \ell_n}{ \ell_1 + \cdots + \ell_n}=\frac{(n_1 + \dots + n_k) \frac{\ln n_k}{n_k}}{ \ln n_1 + \cdots + \ln n_k}.
$$
If $n_k$ grows fast enough, then
$$
{\overline{\lim_{n \rightarrow \infty}}}\frac{n \ell_n}{ \ell_1 + \cdots + \ell_n} \geq \lim_{k \rightarrow \infty}\frac{\frac{n_1}{n_k}\ln n_k+ \dots + \frac{n_{k-1}}{n_k}\ln n_k+ \ln n_k}{\ln n_1 + \cdots + \ln n_k}=1.
$$

(B) \eqref{s2} is not necessary for {\em  \eqref{suff-cond1-}}.

For this we modify the construction above. For all even $k$s, we alter the construction above as follows. For all $n \in \mathbb{N}$, with
$$
n_1 + \dots + n_k \leq n \leq n_1 + \dots + n_{k+1},
$$
we define
$$
\ell_n = \frac{c}{n},
$$
where $c>1$.
Note that 
$$
\frac{\ln n_k}{n_k}\geq \frac{c}{n_1 + \dots + n_k}.
$$
Next we take $n_{k+1}$ so large that 
$$
\sum_{n=1}^{n_1 + \dots  + n_{k+1}}\frac{1}{n^2}e^{a(\ell_1 + \dots +\ell_n)}>k.
$$
This is doable, since $ac \geq 1$.
As a result, we will get a sequence $(\ell_n)_{n \geq 1}$ for which \eqref{s2} fails, but for all $a\geq c$ the series in \eqref{suff-cond1-} diverges.

\section{Proofs of Theorem ~\ref{thm1}  , Corollaries ~\ref{cor1},~\ref{cor2} and Theorem \ref{thm_seq}}
\subsection{Proof of Theorem ~\ref{thm1}}

{\em Case $m_f=0$}. 
By assumption, we have \eqref{dom_1}. Hence, by Proposition \ref{l-1} the condition \eqref{cond01} is necessary and sufficient for covering the set $\mathbb{T}\setminus K_f$.
Obviously, the condition \eqref{cond0} is necessary for covering $K_f$. Otherwise for some $x_0 \in \mathbb{T}$ we will have
$$
\sum_{n=1}^\infty \mu_f(B(x_0, r_n))<\infty,
$$
a contradiction to the fact that $x_0$ is covered. 
Since $K_f$ is countable, then the condition \eqref{cond0} will also be sufficient for covering the set $K_f$. 

{\em Case $m_f>0$}. 
The sufficiency of the
first conditions in \eqref{cond1} follows from Proposition \ref{suff}. We now show its necessity.
As above, the necessity of the first condition in \eqref{cond1} follows from the assumptions \eqref{dom_1}. 

The second fact in \eqref{cond1} implies that $K_f\cap F_f$
and $K_f \cap (a_n + K_f\cap F_f)$ are all covered by Kahane's result and 
Corollary \ref{f-0}. The points in $K_f$, which are not covered by the
translates of $K_f\cap F_n$ are countable, these countable set are covered
according to the condition \eqref{cond0}.  
Observe that when $m_f>0$, the condition \eqref{cond0} is automatically fulfilled. This is because $\mu_f(B(x, r_n)) \ge m_f \ell_n$
and $\sum \ell_n =\infty$. Actually we get more
$$
    \forall x \in \mathbb{T}, \quad \sum_{n=1}^\infty \mu_f(B(x, r_n)) =\infty.
$$
So, we have only to check the  necessity of the second condition in \eqref{cond1} 
which, in turn, follows from Proposition \ref{t-1}.

\subsection{Proof of Corollary ~\ref{cor1}}
\begin{proof}
By assumption, the set of points where we do not have a Lipschitz dominant $g \in \lip(U)$, is countable, hence the condition \eqref{cond0} will be necessary and sufficient for covering this set.
As for all $x \in K_f \cap F_f$, there is a set $U$ with $x \in U$ and a function $g_x=g \in \lip(U)$, so that $f(t)\leq g(t)$ for almost all $t \in U$ and $f(t)=g(t)$, for $t \in K_f \cap U$. Since $g$ is continuous at $x$ and $f(x)=g(x)=m_f$, then $f$ is continuous at $x$ too, so according to Corollary \ref{cont} the first condition in \ref{cond1} is necessary and sufficient for covering the set $K_f^c$. To finish the proof, we need to show that for $\forall x \in K_f$ we have the flatness condition.

Fix $x\in K_f$. Let $C$ be the  Lipschitz constant of $g$.
Hence
$$
\mu(B(x, r_n))\le\int_{x-r_n}^{x+r_n}g(y)dy \leq \int_{x-r_n}^{x+r_n}(C|x-y|+|g(x)|)dy
\le C r_n^2 + \ell_n f(x).
$$
Since $x \in K_f$, then $f(x)=m_f$. Hence
$$
|\mu(B(x, r_n))-m_f \ell_n|\leq C r_n^2.$$
From which
$$
\sum_{n=1}^\infty |\mu(B(x, r_n))-m_f \ell_n|\leq C \sum_{n=1}^\infty r_n^2.
$$

\indent {\rm (1)}
Let $m_f=0$ and $\sum_{n=1}^\infty \ell_n^2<\infty$, then
$$
\sum_{n=1}^\infty \mu(B(x, r_n)) < \infty.
$$
Hence, there will be no $\mu_f$-Dvoretzky covering.

\indent {\rm (2)}
If $m_f>0$, then we can assume
$$
\sum_{n=1}^\infty \ell_n^2 < \infty,
$$
from which
$$
\sum_{n=1}^\infty |\mu(B(x, r_n))-m_f \ell_n|< \infty.
$$

Thus, at $x$ we have the flatness property. To finish the proof we will need to apply Theorem \ref{thm1}.

\end{proof}

\subsection{Proof of Corollary ~\ref{cor2}}
\begin{proof}
Let $K = K_f \cap F_f$. By assumption, $K$  is of positive Lebesgue measure. Suppose $\mathbb{T}$ is covered. Then so is $K$ which supports the Lebesgue measure restricted on $K$,  and by Proposition \ref{t-1},  we have
\begin{equation}
\int_K \int_K \exp{\left( m_f \sum_{k=1}^\infty (\ell_k - |t-s|)_+\right)}d t ds=\infty,
\end{equation}
which is equivalent to the condition \eqref{whole} 
according to Lemma \ref{lem-1}. Thus, the condition \eqref{whole} is necessary for covering the circle. 
The sufficiency of the condition \eqref{whole} for covering the circle follows from  Proposition \ref{c-2}. 
\end{proof}

\subsection{Proof of Theorem \ref{thm_seq}}
\begin{proof}
Observe that $\ell_1 + \dots + \ell_n = c\log n + O(1)$. Hence
\begin{equation}\label{c/n}
\sum_{n=1}^\infty \frac{1}{n^2}e^{a(\ell_1 +\cdots +\ell_n)}= \infty
\Leftrightarrow ac \geq 1.
\end{equation}

Assume 
$
cm_f \geq 1.
$
By   Corollary \ref{c-2} and (\ref{c/n}) applied to $a=m_f$, the circle is covered.  

Conversely assume that the circle is covered.  
Remark that  the condition (\ref{s2}) 
is satisfied. Indeed,
$$
\frac{n \ell_n}{\ell_1 + \dots + \ell_n}=\frac{c}{c \ln n + O(1)}=o(1).
$$
Then by Proposition \ref{cond_seq}, we have
$$
\forall \epsilon >0, \ \ \ \sum_{n=1}^\infty\frac{1}{n^2}e^{(m_f+\epsilon)(\ell_1 + \dots + \ell_n)}=\infty,
$$
which, according to (\ref{c/n}), implies
$c(m_f + \epsilon)\geq 1$ for all $\epsilon>0$. Therefore
$
cm_f \geq 1.
$
\end{proof}


\medskip
\textit{Acknowledgement.}
The second author would like to thank the School of Mathematics and Statistics
of the Central China Normal University for its hospitality during his visit in the Spring semester of 2018, where the work was started.

\bibliographystyle{plain}

\end{document}